\documentclass[12pt]{amsart}
\usepackage[latin1]{inputenc}
\usepackage{amssymb, amsmath, amscd, mathrsfs,marvosym, psfrag,color,a4} 

\input xy
\xyoption{all}
\DeclareMathAlphabet{\mathpzc}{OT1}{pzc}{m}{it}

\newtheorem{theorem}{Theorem}[section]
\newtheorem*{theorem1}{Theorem 1}
\newtheorem*{theorem2}{Theorem 2}

\newtheorem{proposition}[theorem]{Proposition}

\newtheorem{lemma}[theorem]{Lemma}

\theoremstyle{definition}
\newtheorem{definition}[theorem]{Definition}

\theoremstyle{remark}
\newtheorem{remark}[theorem]{Remark}

\newcommand{\CC}{{\mathcal C}}

\newcommand{\CJ}{{\mathcal J}}
\newcommand{\CK}{{\mathcal K}}

\newcommand{\CO}{{\mathcal O}}

\newcommand{\SK}{{\mathscr K}}

\newcommand{\SQ}{{\mathscr Q}}

\newcommand{\SZ}{{\mathscr Z}}

\newcommand{\DZ}{{\mathbb Z}}

\newcommand{\DN}{{\mathbb N}}
\newcommand{\DQ}{{\mathbb Q}}

\newcommand{\ch}{{\operatorname{char}\, }}

\newcommand{\End}{{\operatorname{End}}}

\newcommand{\rad}{{\operatorname{rad}}}

\newcommand{\height}{{\operatorname{ht}}}

\newcommand{\im}{{\operatorname{im}\,}}

\newcommand{\rk}{{{\operatorname{rk}}}}

\newcommand{\ul}{\underline}

\newcommand{\id}{{\operatorname{id}}}

\newcommand{\Quot}{{\operatorname{Quot\,}}}

\newcommand{\comment}[1]{}

\newcommand{\lgl}{\langle}
\newcommand{\rgl}{\rangle}

\newcommand{\qchoose}[2]{\left[{#1\atop#2}\right]}

\begin{document}

\pagenumbering{arabic}
\title[]{Periodicity of irreducible modular and quantum characters} \author[]{Peter Fiebig}
\begin{abstract} For a root system $R$, a field $\SK$ and an invertible element $q$ in $\SK$ let  $U_{(\SK,q)}(R)$ be the associated quantum group, defined via Lusztig's divided powers construction. We study the irreducible characters of this algebra with integral (but not necessarily dominant) highest weight. If $\sigma_l(q)=0$, where $\sigma_l$ is the  $l$-th cyclotomic polynomial, then these characters exhibit a certain $l$-periodicity. 
 \end{abstract}

\address{Department Mathematik, FAU Erlangen--N\"urnberg, Cauerstra\ss e 11, 91058 Erlangen}
\email{fiebig@math.fau.de}
\maketitle

\section{Introduction}
Let $R$ be a root system, $\SK$ a field and $q\in\SK$ an invertible element. We denote by $U_\SK=U_{(\SK,q)}(R)$ the corresponding quantum group that is obtained via base change from Lusztig's integral form $U_{\DZ[v,v^{-1}]}$.  This is a triangularized algebra, i.e. we have a decomposition  $U_\SK=U_\SK^-\otimes U_\SK^0\otimes U_\SK^+$. Let $X$ be the weight lattice of $R$. Any  $\lambda\in X$ gives rise to a character $ U_\SK^0\to\SK$ (of ``type 1''), and the standard construction (via Verma modules) yields  the irreducible $U_\SK$-module  $L_\SK(\lambda)$ of highest weight $\lambda$. It has a weight decomposition  $L_\SK(\lambda)=\bigoplus_{\mu\in X} L_\SK(\lambda)_\mu$. For $l>0$ denote by $\sigma_l\in\DZ[v]$ the $l$-th cyclotomic polynomial.
The main result of this article is the following.

\begin{theorem1} Let $\lambda,\mu\in X$ and assume that $\mu\le\lambda$. Suppose that $l\ge 1$ is strictly bigger than any coefficient in the expansion of $\lambda-\mu$ as a sum of simple roots, and that  
 $\sigma_l(q)=0$ in $\CK$. 
Then
$$
\dim L_\SK(\lambda)_\mu=\dim L_\SK(\lambda+l\gamma)_{\mu+l\gamma}
$$
for all $\gamma\in X$. 
\end{theorem1}

Suppose that  $q=1$.  If $\ch \SK=p>0$, then  $\sigma_{l}(1)=0$ in $\SK$ if and only if $l=p^r$ for some  $r\ge0$. So in this case the above result specializes to the following.  

\begin{theorem2} Suppose that $\ch \SK=p>0$ and $q=1$. Let $\lambda,\mu\in X$ and suppose that $\mu\le\lambda$. Suppose that $r\ge0$ is such that $p^r$ is strictly bigger than any coefficient in an expansion of $\lambda-\mu$ in terms of simple roots. Then
$$
\dim L_\SK(\lambda)_\mu=\dim L_\SK(\lambda+p^r\gamma)_{\mu+p^r\gamma}
$$
for all $\gamma\in X$.
\end{theorem2}
Note that in this case the periodicity has a certain {\em fractal} nature.  Note also that for $r=1$ one can  find the statement of Theorem 2  somewhat hidden in the proof of the theorem in  the appendix of \cite{JanMhG}. Jantzen's  arguments  resemble the arguments that we present in this article in some respects.

For $q=1$ and arbitrary $\SK$ there is a surjective homomorphism $U_\SK\to D_\SK$, where $D_\SK$ is the hyperalgebra of the split semisimple, simply connected algebraic group $G_\SK$ with root system $R$. For a dominant weight $\lambda$ we obtain $L_\SK(\lambda)$ by restriction  from the irreducible rational representation of $G_\SK$ with highest weight $\lambda$. 
For dominant weights  $\lambda$ and $\lambda+p^r\gamma$ the statement of Theorem 2 is implied by Steinberg's tensor product theorem for irreducible representations of $G_\SK$. 

Given arbitrary $\lambda,\mu$ we can find $r\gg 0$ that satisfies the assumption in Theorem 2, and then some dominant weight $\gamma$ such that $\lambda+p^r\gamma$ is dominant. So  the characters of all $L_\SK(\lambda)$ can be deduced from the set of  dominant characters. Steinberg's tensor product theorem shows that the dominant characters can be deduced from the (finite) set of restricted dominant characters.  So even in this more general setup, the restricted irreducible characters contain sufficient information to determine all irreducible characters.

The main idea of the proof of Theorem 1 is  to realize $L_\SK(\lambda)$ as the quotient of an auxiliary object $P_\SK(\lambda)$, which is an $X$-graded vector space freely generated by the set of {\em simple root paths} starting at $\lambda$. The quotient is taken by the radical of a bilinear form $b_\lambda$.  Note that a pair $(\SK,q)$ as above as a $\SZ=\DZ[v,v^{-1}]$-algebra $\SK$ that is a field.  The objects $P_\SK(\lambda)$ and  $b_\lambda$ can be obtained from integral versions over $\SZ$ by base change. But the radical, of course,  depends heavily on $\SK$ and $q$.  We study the $\SZ$-matrix of $b_\lambda$ in the basis of simple root paths. Each entry is given by evaluating a {\em quantum polynomial function} $X\to\DZ[v,v^{-1}]$ at $\lambda$. These functions are polynomials in the quantized linear functions $[f(\cdot)]\colon X\to \SZ$, but with coefficients  in the quotient field $\SQ:=\DQ(v)$. It is not  difficult to control the denominators, and a simple arithmetic statement about the periodicity of quantum polynomial functions then yields the periodicity of Theorem 1.

{\bf Acknowledgement:} This material is based upon work supported by a grant of the Institute for Advanced Study School of Mathematics.

\section{The quantum group $U_\SK$}
Let $R$ be a root system and $\Pi\subset R$ a basis. For any $\alpha\in R$ we denote by $\alpha^\vee$ its coroot. Set $a_{\alpha,\beta}=\lgl\alpha,\beta^\vee\rgl$ for $\alpha,\beta\in\Pi$. Then there is a vector $d=(d_\alpha)$ with entries $d_\alpha\in\{1,2,3\}$ such that $(d_\alpha a_{\alpha,\beta})$ is a symmetric matrix and each irreducible component of $R$ contains an $\alpha$ with $d_\alpha=1$. 
The goal of this section is to define the quantum group $U_{(\SK,q)}$ and its category $\CO_{(\SK,q)}$ of highest weight representations for any field $\SK$ and any non-zero element $q$ in $\SK$. It has a triangular decomposition and a category $\CO$ of representations. 

\subsection{The (integral) quantum group}
Set
$\SZ:=\DZ[v,v^{-1}]$.  For $n\in\DZ$   define
$$
[n]:=\frac{v^{n}-v^{-n}}{v-v^{-1}}=
\begin{cases}
0,&\text{ if $n=0$,}\\
v^{n-1}+v^{n-3}+\dots+v^{-n+1},&\text{ if $n> 0$,}\\
-v^{-n-1}-v^{-n-3}-\dots-v^{n+1},&\text{ if $n<0$}.
\end{cases}
$$
For $n\ge 0$ set $[n]^!=[1]\cdot[2]\cdots[n]$, and for $n\in\DZ$ and $r\ge 0$ set
$$
\qchoose{n}{r}:=\frac{[n]\cdot[n-1]\cdots[n-r+1]}{[1]\cdot[2]\cdots[r]}.
$$
For $\alpha\in\Pi$ set  $v_\alpha:=v^{d_\alpha}$ and 
$$
[n]_\alpha:=[n](v_\alpha)=\frac{v_\alpha^n-v_\alpha^{-n}}{v_\alpha-v_{\alpha}^{-1}}.
$$
Define $[n]_\alpha^!$, $\qchoose{n}{r}_\alpha\in\SZ$ likewise. All of the above are elements in $\SZ$.

Set  $\SQ:=\DQ(v)=\Quot(\SZ)$.  The {\em quantum group $U_\SQ=U_\SQ(R)$} associated with $R$ is the associative, unital $\SQ$-algebra with set of generators $\{e_\alpha, f_\alpha, k_\alpha, k_\alpha^{-1}\mid\alpha\in\Pi\}$ subject to the following relations (cf. \cite[Section 1.1]{Lus90}):

\begin{align*}
k_\alpha k_\beta&=k_\beta k_\alpha,  & k_\alpha k_\alpha^{-1}&=1=k_\alpha^{-1}k_\alpha,\\
k_\alpha e_\beta&=v_\alpha^{a_{\alpha\beta}} e_\beta k_\alpha, & k_\alpha f_\beta&=v_\alpha^{a_{\alpha\beta}} f_\beta k_\alpha,\\
[e_\alpha,f_\beta]&=0\quad\text{ ($\alpha\ne\beta$)},& [e_\alpha,f_\alpha]&=\frac{k_\alpha-k_{\alpha}^{-1}}{v_\alpha -v_\alpha^{-1}}\\
\end{align*}
and the {\em Serre-relations}
\begin{align*}0&=\sum_{r+s=1-a_{\alpha\beta}}(-1)^s\qchoose{1-a_{\alpha\beta}}{s}_{\alpha}e^r_\alpha e_\beta e_\alpha^s\quad\text{ for $\alpha\ne\beta$},\\
0&=\sum_{r+s=1-a_{\alpha\beta}}(-1)^s\qchoose{1-a_{\alpha\beta}}{s}_{\alpha}f^r_\alpha f_\beta f_\alpha^s\quad\text{ for $\alpha\ne\beta$}.
\end{align*}
We endow $U_\SQ$ with an  $X$-grading such that $\deg(e_\alpha)=\alpha$, $\deg(f_\alpha)=-\alpha$ and $\deg(k_{\alpha})=\deg(k_\alpha^{-1})=0$. 

For $\alpha\in\Pi$ and $n\ge 0$ define
$$
e_{\alpha}^{[n]}:=e_\alpha^n/[n]_{\alpha}^!, \quad f_{\alpha}^{[n]}:=f_\alpha^n/[n]_{\alpha}^!,\quad
\qchoose{k_\alpha}{n}_\alpha:=\prod_{s=1}^n\frac{k_\alpha v_\alpha^{-s+1}-k_\alpha^{-1}v_\alpha^{s-1}}{v_\alpha^{s}-v_\alpha^{-s}}.
$$
 Following ideas of Kostant, Lusztig defined the integral version $U_\SZ$  of $U_\SQ$ as the $\SZ$-subalgebra generated by the set $\{e_{\alpha}^{[n]},f_\alpha^{[n]}\mid \alpha\in \Pi, n\ge 0\}\cup\{k_\alpha,k_\alpha^{-1}\mid\alpha\in\Pi\}$. Then $U_\SZ^+$, $U_\SZ^-$ and $U_\SZ^0$ are defined as the $\SZ$-subalgebras in $U_\SZ$ generated by $e_{\alpha}^{[n]}$'s, the $f_\alpha^{[n]}$'s, and the set $\{k_\alpha,k_\alpha^{-1},\qchoose{k_\alpha}{n}_\alpha\mid \alpha\in\Pi, n\ge 0\}$, resp. (
cf. \cite[Section 1.3 \& 1.4]{Lus90}). All the algebras above inherit an $X$-grading from $U_\SQ$.

 In  \cite[Section 6.5]{Lus90} the elements $e_\alpha^{[n]}$ and $f_{\alpha}^{[n]}$ are defined for each positive root $\alpha\in R^+$ and $n\ge0$. Then we have the following PBW-type theorem. 
 
\begin{theorem}[{\cite[Theorem 6.7]{Lus90}}] \label{thm-PBW}
\begin{enumerate}
\item The algebra $U_\SZ^+$ is free as  a $\SZ$-module, and the elements $\prod_{\alpha\in R^+}e_\alpha^{[n_\alpha]}$ with $n_\alpha\in \DZ_{\ge 0}$ form a basis.
\item The algebra $U_\SZ^-$ is free as  a $\SZ$-module, and the elements $\prod_{\alpha\in R^+}f_\alpha^{[n_\alpha]}$ with $n_\alpha\in \DZ_{\ge 0}$ form a basis.
\item The algebra $U_\SZ^0$ is free as a  $\SZ$-module, and the elements $\prod_{\alpha\in \Pi}k_{\alpha}^{\delta_\alpha}\qchoose{k_\alpha}{r_\alpha}_\alpha$ with $\delta_\alpha\in\{0,1\}$ and $r_\alpha\in \DZ_{\ge 0}$ form a basis.
\item The multiplication defines an isomorphism 
$$
U_\SZ^-\otimes U_\SZ^0\otimes U_\SZ^+\xrightarrow{\sim} U_\SZ
$$
of $\SZ$-modules. 
\end{enumerate}
\end{theorem}
Note that the products in parts (1) and (2) in the theorem above have to be taken with respect to  certain orders on the set of positive roots, cf. Theorem \cite[Theorem 6.7]{Lus90}.

Fix a field $\SK$ and an invertible element  $q$ in $\SK$. Then $\SK$ acquires an $\SZ$-algebra structure such that $v\in\SZ$ acts as multiplication with $q$. We set 
$$
U_\SK=U_{(\SK,q)}:=U_\SZ\otimes_\SZ\SK
$$ 
and define  $U_\SK^+$, $U_\SK^-$ and $U_\SK^0$ likewise. Theorem \ref{thm-PBW} implies that these are subalgebras in $U_\SK$ and that  the multiplication map $U_\SK^-\otimes U_\SK^0\otimes U_\SK^+\to U_\SK$ is an isomorphism. 

\subsection{Representation theory of $U_\SK$} We denote by $X$ the weight lattice associated with $R$.
By \cite[Lemma 1.1]{APW},  every $\mu\in X$ yields a character (denoted by the same letter)
\begin{align*}
\mu\colon U^0_\SZ&\to\SZ\\
k_\alpha^{\pm1}&\mapsto v_\alpha^{\pm \lgl\mu,\alpha^\vee\rgl}\\
\qchoose{k_\alpha}{r}_\alpha&\mapsto \qchoose{\lgl\mu,\alpha^\vee\rgl}{r}_{\alpha}\text{ ($\alpha\in\Pi$, $r\ge0$)}.
\end{align*}
A $U_\SK$-module $M$ is called a {\em weight module} if $M=\bigoplus_{\mu\in X}M_\mu$, where
$$
M_\mu:=\{m\in M\mid H.m=\mu(H)m\text{ for all $H\in U_\SK^0$}\}.
$$
By definition, all our weight modules are ``of type 1'' (cf. \cite[Section 5.1]{JanQG}).
 
\begin{remark}\label{rem-dequant} Suppose that $\CK$ is a field and $q=1$.
There is a surjective homomorphism $U_\SK\to D(G_\SK)$, where $G_\SK$ is the split semisimple, simply connected and connected algebraic group with root system $R$ over  $\SK$, and $D(G_\SK)$ is the algebra of distributions of $G_\SK$. The homomorphism maps $e_\alpha$ and $f_\alpha$ to the Serre generators of $D(G_\SK)$,  $\qchoose{k_\alpha}{r}_\alpha$ to ${h_\alpha}\choose{r}$ for all $\alpha\in\Pi$, while $k_\alpha$ is mapped to $1$ (cf. Theorem \ref{thm-PBW}, \cite[Section 8.15]{Lus90}). By \cite[Section II, 1.20]{JanAG}, a finite dimensional representation of $U_\SK$ is the same as a rational representation of $G_\SK$-module. 
\end{remark}

The following (well-known) relation is a cornerstone of the approach towards representations of quantum groups chosen in this article. 
 \begin{lemma} \label{lemma-comrelsU} Let $M$ be a $U_\SK$-module. For all $\alpha,\beta\in\Pi$ and $v\in M_\mu$ we then have
 $$
 e_\alpha^{[m]}f_\beta^{[n]}(v)=
 \begin{cases}
 f_\beta^{[n]} e_\alpha^{[m]}(v),&\text{ if $\alpha\ne\beta$},\\
 \sum_{r=0}^{\min(m,n)}\qchoose{\lgl\mu,\alpha^\vee\rgl+m-n}{r}_\alpha f_\beta^{[n-r]}e_\alpha^{[m-r]}(v),&\text{ if $\alpha=\beta$}.
 \end{cases}
 $$
 \end{lemma}
 \begin{proof} The case $\alpha\ne\beta$ follows directly from the fact that $e_\alpha$ and $f_\beta$ commute. By \cite[Section 6.5]{Lus90} we have the following relation in the algebra $U_\SZ$:
 $$
 e_{\alpha}^{[m]}f_\beta^{[n]}=\sum_{r=0}^{\min(m,n)} f_\alpha^{[n-r]}\qchoose{k_\alpha;2r-m-n}{r}_\alpha e_\alpha^{[m-r]},
 $$
 where
$$
\qchoose{k_\alpha;c}{r}_\alpha =\prod_{s=1}^r\frac{k_\alpha v_\alpha^{c-(s-1)}-k_{\alpha}^{-1}v_\alpha^{-(c-(s-1))}}{v_\alpha^{s}-v_\alpha^{-s}}.
$$
This is an element in $U_\SK^0$  that acts as multiplication with 
\begin{align*}
\prod_{s=1}^r\frac{v_\alpha^{\lgl\nu,\alpha^\vee\rgl+c-(s-1)}-v_\alpha^{-(\lgl\nu,\alpha^\vee\rgl+c-(s-1))}}{v_\alpha^{s}-v_\alpha^{-s}}.
\end{align*}
on a vector of weight $\nu$.
For  $v\in M_\mu$ we obtain
\begin{align*}
 e_{\alpha}^{[m]}f_\beta^{[n]}(v)&=\sum_{r=0}^{\min(m,n)} f_\alpha^{[n-r]}\prod_{s=1}^r\frac{v_\alpha^{\zeta-(s-1)}-v_\alpha^{-(\zeta-(s-1))}}{v_\alpha^{s}-v_\alpha^{-s}}e_\alpha^{[m-r]}(v),
 \end{align*}
 where $\zeta=\lgl\mu+(m-r)\alpha,\alpha^\vee\rgl+2r-m-n=\lgl\mu,\alpha^\vee\rgl+m-n$. It remains to show that
 $$
 \qchoose{\zeta}{r}_\alpha=\prod_{s=1}^r\frac{v_\alpha^{\zeta-(s-1)}-v_\alpha^{-(\zeta-(s-1))}}{v_\alpha^{s}-v_\alpha^{-s}},
 $$
 which is (almost) immediate from the definition. 
\end{proof}
 \subsection{Category $\CO$}
 Given the triangular decomposition of $U_\SK$, the following is a standard definition.
 \begin{definition}
  We denote by $\CO_\SK=\CO_{(\SK,q)}$ the full subcategory of the category of  $U_\SK$-modules that contains all  $M$ that are weight modules and have the property that  for any $m\in M$  the $\SK$-vector space $U_\SK^{+}.m$ is finite dimensional. 
  \end{definition}
  Clearly,  $\CO_\SK$  is an abelian category. Note that in the case $\ch \SK=0$ the above category is studied in \cite{AM}.

   Set $U_\SK^{\ge 0}:=U_\SK^0 U_\SK^+\subset U_\SK$. 
 For $\lambda\in X$ we define the {\em Verma module} with highest weight $\lambda$ by
 $$
 \Delta_\SK(\lambda):=U_\SK\otimes_{U_\SK^{\ge 0}}\SK_\lambda,
 $$
 where $\SK_\lambda$ is the one-dimensional $\SK$-vector space, endowed with the $U_\SK^{\ge 0}$-action such that $U_\SK^0$ acts via the character $\lambda$, and $e_{\alpha}^{[n]}\in U_\SK^{+}$ acts trivially for $n>0$ (this is possible by Theorem \ref{thm-PBW}).  Standard arguments show that $\Delta_\SK(\lambda)$ is an object in $\CO_\SK$. Standard arguments also show that the following holds (cf. \cite[Section 5.5]{JanQG}).   
 \begin{lemma} For all $\lambda\in X$, the Verma module $\Delta_\SK(\lambda)$ has a unique irreducible quotient $L_\SK(\lambda)$ in $\CO_\SK$. The $L_\SK(\lambda)$ with $\lambda$ in $X$ form a full set of representatives of the irreducible objects in $\CO_\SK$.
 \end{lemma}

\section{A model for the semisimple objects in $\CO_\SK$}
In this section we want to study a category that resembles the category $\CO_\SK$, but ``ignores the Serre-relations''. We will see later that it is equivalent to the full subcategory of $\CO_\SK$ that contains all semisimple objects. 

We consider $X$-graded $\SK$-vector spaces $M=\bigoplus_{\mu\in X}M_\mu$ endowed with $\SK$-linear operators 
$E_{\alpha,n},F_{\alpha,n}\colon M\to M$ for all $\alpha\in\Pi$ and $n>0$ that are homogeneous of degree $+n\alpha$ and $-n\alpha$, resp. We will always set $E_{\alpha,0}=F_{\alpha,0}=\id_{M_\mu}$.
 There are  three conditions that we will assume on the above data. The first two are easy to formulate.
 
\begin{enumerate}
\item[(C1)] There exists some $\gamma\in X$ such that $M_\mu\ne 0$ implies $\mu\le\gamma$.
\item[(C2)] For all $\mu\in X$, $\alpha,\beta\in\Pi$, $m,n>0$, and $v\in M_{\mu}$,
$$
E_{\alpha,m}F_{\beta,n}(v)=\begin{cases}
F_{\beta,n}E_{\alpha,m}(v),&\text{if $\alpha\ne\beta$,}\\
\sum_{r= 0}^{\min(m,n)} \qchoose{\lgl\mu,\alpha^\vee\rgl+m-n}{r}_{\alpha}F_{\alpha,n-r}E_{\alpha,m-r}(v),&\text{if $\alpha=\beta$}
\end{cases}
$$
(cf. Lemma \ref{lemma-comrelsU}).
\end{enumerate}
(cf. Lemma \ref{lemma-comrelsU}.) A remark on property (C1): Note that for any finite subset $T$ of $X$ there exists some $\gamma\in X$ such that $\mu\le\gamma$ for all $\mu\in T$.

In order to formulate the third condition, we need some definitions. For any $\mu\in X$ define
$$
M_{\delta\mu}:=\bigoplus_{\alpha\in\Pi, n>0} M_{\mu+n\alpha}
$$
and  let 
\begin{align*}
E_\mu&\colon M_\mu\to M_{\delta\mu},\\
F_\mu&\colon M_{\delta\mu}\to M_\mu
\end{align*}
be the column and the row vector with entries $E_{\alpha,n}|_{M_\mu}$ and $F_{\alpha,n}|_{M_{\mu+n\alpha}}$, resp.  Here is the third condition.

\begin{enumerate}
\item[(C3)] For all $\mu\in X$ we have $M_\mu=\ker E_\mu\oplus\im F_\mu$. 
\end{enumerate}
Now we can define the model category.
\begin{definition} The category $\CC_{\SK}=\CC_{(\SK,q)}$ is defined as follows. Its objects are the $X$-graded $\SK$-vector spaces  $M=\bigoplus_{\mu\in X}M_\mu$ endowed with $\SK$-linear endomomorphisms $E_{\alpha,n}, F_{\alpha,n}\colon M\to M$ of degree $+n\alpha$ and $-n\alpha$, resp., for all $\alpha\in\Pi$ and $n>0$, for which the conditions (C1), (C2) and (C3) are satisfied. A morphism $f\colon M\to N$  in $\CC_{\SK}$ is a degree $0$ homogeneous $\SK$-linear map  that commutes with all $E$'s and $F$'s.\end{definition}
For a morphism $f\colon M\to N$ in $\CC_\SK$ we denote by $f_\mu\colon M_\mu\to N_\mu$ its $\mu$-component.
Let us already formulate one of the main results of this article.

\begin{theorem} \begin{enumerate}
\item Each object $M$ in $\CC_\SK$ carries a unique $U_\SK$-module structure such that the following holds.
\begin{enumerate}
\item The $X$-grading on $M$ is the weight space decomposition.
\item For $\alpha\in\Pi$ and $n>0$, the homomorphisms $E_{\alpha,n}$ and $F_{\alpha,n}$ are the action maps of $e_{\alpha}^{[n]}$ and $f_{\alpha}^{[n]}$, resp.
\end{enumerate}
\item With the $U_\SK$-module structure defined by (1), each object $M$ of $\CC_\SK$ is a semisimple object in $\CO_\SK$ of finite rank. 
\item We obtain an equivalence of $\CC_\SK$ with the full subcategory of $\CO_\SK$ that contains all semisimple ojects of finite rank.
\end{enumerate}
\end{theorem}
\subsection{Extending morphisms} Let us call a subset $\CJ$ of $X$  {\em open} if $\lambda\in\CJ$ and $\lambda\le\mu$ imply $\mu\in\CJ$.
A useful property of the category $\CC_\SK$ is that ``partial morphisms'', i.e. morphisms  that are defined only on open subsets, extend. 

\begin{lemma}\label{lemma-extmor} Let $M$ and $N$ be objects in $\CC_\SK$, and let $\CJ$ be an open subset of $X$.  Suppose that for all $\nu\in \CJ$ we have a $\SK$-linear homomorphism $f_{(\nu)}\colon M_\nu\to N_\nu$ such that the diagrams 

\centerline{
\xymatrix{
M_{\nu+n\alpha}\ar[r]^{f_{(\nu+n\alpha)}}\ar[d]_{F_{\alpha,n}}&N_{\nu+n\alpha}\ar[d]^{F_{\alpha,n}}\\
M_\nu\ar[r]^{f_{(\nu)}}&N_\nu
}
\quad\quad
\xymatrix{
M_{\nu+n\alpha}\ar[r]^{f_{(\nu+n\alpha)}}&N_{\nu+n\alpha}\\
M_{\nu}\ar[u]^{E_{\alpha,n}}\ar[r]^{f_{(\nu)}}&N_{\nu}\ar[u]_{E_{\alpha,n}}
}
}
\noindent
commute for all $\nu\in \CJ$, $\alpha\in\Pi$ and $n>0$.
Then there exists a morphism $f\colon M\to N$ in $\CC_\SK$ such that $f_\nu=f_{(\nu)}$ for all $\nu\in\CJ$. 
\end{lemma}
\begin{proof} By definition there is some $\gamma$ such that $M_\mu\ne 0$ or $N_\mu\ne 0$ implies $\mu\le\gamma$. We can hence enlarge $\CJ$ by adding all weights $\nu\not\le\gamma$ and setting $f_{(\nu)}=0$ for those $\nu$. After that we can use an inductive argument and need to show the following. If $\mu\in X$ is such that $\mu\not\in\CJ$ and $\CJ^\prime:=\CJ\cup\{\mu\}$ is open again, then there exists a homomorphism $f_{(\mu)}\colon M_\mu\to N_\mu$ such that the diagrams 

\centerline{
\xymatrix{
M_{\mu+n\alpha}\ar[r]^{f_{(\mu+n\alpha)}}\ar[d]_{F_{\alpha,n}}&N_{\mu+n\alpha}\ar[d]^{F_{\alpha,n}}\\
M_\mu\ar[r]^{f_{(\mu)}}&N_\mu
}
\quad\quad
\xymatrix{
M_{\mu+n\alpha}\ar[r]^{f_{(\mu+n\alpha)}}&N_{\mu+n\alpha}\\
M_{\mu}\ar[u]^{E_{\alpha,n}}\ar[r]^{f_{(\mu)}}&N_{\mu}\ar[u]_{E_{\alpha,n}}
}
}
\noindent
commute for all $\alpha\in\Pi$ and $n>0$.

Let $\alpha,\beta\in\Pi$ and $m,n>0$. We now show  that the  diagram 

\centerline{
\xymatrix{
M_{\mu+n\beta}\ar[d]_{f_{(\mu+n\beta)}}\ar[r]^{F_{\beta,n}}&M_\mu\ar[r]^{E_{\alpha,m}}&M_{\mu+m\alpha}\ar[d]^{f_{(\mu+m\alpha)}}\\
N_{\mu+n\beta}\ar[r]^{F_{\beta,n}}&N_\mu\ar[r]^{E_{\alpha,m}}&N_\mu
}
}
\noindent commutes.
Let $v\in M_{\mu+n\beta}$. Using (C2) we obtain some scalars $c_r\in\SK$ such that 
\begin{align*}
f_{(\mu+m\alpha)} E_{\alpha,m}F_{\beta,n}(v)
&=f_{(\mu+m\alpha)}\sum_{0= r}^{\min(m,n)} c_r F_{\beta,n-r}E_{\alpha,m-r}(v) \\
&=\sum_{0= r}^{\min(m,n)} c_r f_{(\mu+m\alpha)}F_{\beta,n-r}E_{\alpha,m-r}(v)\\
&=\sum_{0= r}^{\min(m,n)} c_r F_{\beta,n-r}f_{(\mu+m\alpha+(n-r)\beta)}E_{\alpha,m-r}(v)\\
&=\sum_{0= r}^{\min(m,n)} c_rF_{\beta,n-r}E_{\alpha,m-r}f_{(\mu+n\beta)}(v)\\
&=E_{\alpha,m}F_{\beta,n}f_{(\mu+n\alpha)}(v)
\end{align*}
(note that $f_{(\mu+n\beta)}(v)$ also has weight $\mu+n\beta$).
It now follows that the diagram 

\centerline{
\xymatrix{
M_{\delta\mu}\ar[d]_{f_{(\delta\mu)}}\ar[r]^{F_{\mu}}&M_\mu\ar[r]^{E_{\mu}}&M_{\delta\mu}\ar[d]^{f_{(\delta\mu)}}\\
N_{\delta\mu}\ar[r]^{F_{\mu}}&N_\mu\ar[r]^{E_\mu}&N_{\delta\mu}
}
}
\noindent commutes.
By (C3), the homomorphisms $E_\mu$ are injective on the image of $F_\mu$, so we obtain a unique induced homomorphism $\tilde f_{(\mu)}$ such that the diagram

\centerline{
\xymatrix{
M_{\delta\mu}\ar[d]_{f_{(\delta\mu)}}\ar[r]^{F_{\mu}}& \im F_\mu\ar[d]^{\tilde f_{(\mu)}}\ar[r]^{E_{\mu}}&M_{\delta\mu}\ar[d]^{f_{(\delta\mu)}}\\
N_{\delta\mu}\ar[r]^{F_{\mu}}&\im F_\mu\ar[r]^{E_\mu}&N_{\delta\mu}
}
}
\noindent commutes. By property (C3) we have $M_\mu=\ker E_\mu\oplus\im F_\mu$ and $N_\mu=\ker E_\mu\oplus\im F_\mu$. One now checks easily that
$f_{(\mu)}:=\left(\begin{matrix}0&0\\0&\tilde f_{(\mu)}\end{matrix}\right)\colon M_\mu\to N_\mu$ serves our purpose. 
\end{proof}

\section{Root paths and bilinear forms}
The goal of this section is to construct for any $\lambda\in X$ an object $S_\SK(\lambda)$ in $\CC_\SK$. We obtain $S_\SK(\lambda)$ as the quotient of an object $P_\SK(\lambda)$ with $E$- and $F$-operators that is freely generated by root paths by the radical of a bilinear form. Note that, in general, $P_\SK(\lambda)$ is not an object in $\CC_\SK$, it won't satisfy condition (C3).

\subsection{Simple root paths}

 A {\em  simple root path}, or simply a {\em root path}, is a sequence $\ul\delta:=(\delta_1,\dots,\delta_l)$ of (not necessarily distinct) simple roots $\delta_1,\dots,\delta_l\in\Pi$ with $l\ge 0$.
We denote by $l(\ul\delta)=l$ its  {\em length} and by $\height(\ul\delta):=\delta_1+\dots+\delta_l$  its {\em height}.  
For $i\in\{1,\dots,l\}$ we set
$$
\ul\delta^{(i)}:=(\delta_1,\dots,\widehat\delta_{i},\dots,\delta_l)
$$ 
(delete the $i$-th entry). This is a path of length $l-1$ and we have $\height(\ul\delta^{(i)})=\height(\ul\delta)-\delta_i$. We denote by $\ul\emptyset$ the path of length $0$.

Fix $\lambda\in X$. For $\mu\in X$ we denote by $P_\SQ(\lambda)_\mu$ the $\SQ$-vector space with basis $\{\ul\delta\mid\height(\ul\delta)=\lambda-\mu\}$, and we set 
$P_\SQ(\lambda):=\bigoplus_{\mu\in X}P_\SQ(\lambda)_\mu$. 
Let $\alpha\in\Pi$. We now define $\SQ$-linear homogeneous operators $\epsilon_\alpha,\varphi_\alpha$ on $P_\SQ(\lambda)$ of degree $+\alpha$ and $-\alpha$. For a path $\ul\delta=(\delta_1,\dots,\delta_l)$ set
\begin{align*}
\varphi_\alpha(\ul\delta)&:=(\alpha,\delta_1,\dots,\delta_l),\\
\epsilon_{\alpha}(\delta_1,\dots,\delta_l)&:=\sum_{i, \delta_i=\alpha}[\lambda-\delta_{i+1}-\dots-\delta_l,\alpha^\vee]_\alpha\ul\delta^{(i)}.
\end{align*}
Here, and in the following, we simplify notation slightly and write $[\nu,\alpha^\vee]_\alpha$ instead of $[\lgl\nu,\alpha^\vee\rgl]_\alpha$.
We define, for $n\ge0$,
$$
\epsilon_{\alpha}^{[n]}:=\epsilon_{\alpha}^n/[n]^!_\alpha,\quad \varphi_{\alpha}^{[n]}:=\varphi_{\alpha}^n/[n]^!_\alpha.
$$
These are $\SK$-linear operators on $P_\SQ(\lambda)$ of degree $+n\alpha$ and $-n\alpha$, resp. 

\subsection{Commutation relations} We now show that the operators defined above satisfy the relation in assumption (C2). 

\begin{lemma} \label{lemma-comrels} Let $\alpha,\beta\in\Pi$ and $m,n>0$. For all $\mu\in X$ and $v\in P_\SQ(\lambda)_\mu$ we have
$$
\epsilon_{\alpha}^{[m]}\varphi_\beta^{[n]}(v)=
\begin{cases}
\varphi_\beta^{[n]}\epsilon_\alpha^{[m]}(v),&\text{ if $\alpha\ne\beta$}\\
\sum_{r=0}^{\min(m,n)}\qchoose{\lgl\mu,\alpha^\vee\rgl+m-n}{r}_\alpha\varphi_{\alpha}^{[n-r]}\epsilon_{\alpha}^{[m-r]}(v),&\text{ if $\alpha=\beta$}.
\end{cases}
$$
\end{lemma}

\begin{proof}
First suppose that $\alpha\ne\beta$. In this case it suffices to consider the case $m=n=1$. Let $\ul\delta=(\delta_1,\dots,\delta_l)$ be a path. Then
\begin{align*}
\epsilon_{\alpha}\varphi_\beta(\ul\delta)&=\epsilon_\alpha(\beta,\delta_1,\dots,\delta_l)\\
&=\sum_{i,\delta_i=\alpha}[\lambda-\delta_{i+1}-\dots-\delta_l,\alpha^\vee]_\alpha(\beta,\delta_1,\dots,\widehat \delta_i,\dots,\delta_l) \\
&=\sum_{i,\delta_i=\alpha}[\lambda-\delta_{i+1}-\dots-\delta_l,\alpha^\vee]_\alpha\varphi_\beta(\ul\delta^{(i)})\\ 
&=\varphi_{\beta}(\sum_{i,\delta_i=\alpha}[\lambda-\delta_{i+1}-\dots-\delta_l,\alpha^\vee]_\alpha\ul\delta^{(i)})\\
&=\varphi_\beta\epsilon_\alpha(\ul\delta).
\end{align*}

Now suppose that $\alpha=\beta$. We first consider the case $m=n=1$. For a path $\ul\delta$ of height $\lambda-\mu$ we have
 \begin{align*}
\epsilon_{\alpha}\varphi_\alpha(\ul\delta)&=\epsilon_\alpha(\alpha,\delta_1,\dots,\delta_l)\\
&=[\lambda-(\lambda-\mu),\alpha^\vee]_\alpha(\delta_1,\dots,\delta_l)+\\
&\quad+\sum_{i,\delta_i=\alpha}[\lambda-\delta_{i+1}-\dots-\delta_l,\alpha^\vee]_\alpha(\alpha,\delta_1,\dots,\widehat{\delta_i},\dots,\delta_l)\\
&=[\mu,\alpha^\vee]_\alpha\ul\delta+\varphi_{\alpha}(\sum_{i,\delta_i=\alpha}[\lambda-\delta_{i+1}-\dots-\delta_l,\alpha^\vee]_\alpha(\delta_1,\dots,\widehat{\delta_i},\dots,\delta_l))\\
&=[\mu,\alpha^\vee]_\alpha\ul\delta+\varphi_{\alpha}\epsilon_\alpha(\ul\delta).
\end{align*}
This is the claimed equation in the case $m=n=1$.

Now let $M=\bigoplus_{l\in\DZ}M_l$ be the $\DZ$-graded $\SQ$-vector space with 
$$
M_l:=\bigoplus_{\mu\in X,\lgl\mu,\alpha^\vee\rgl=l} P_\SQ(\lambda)_\mu.
$$
Then $\epsilon_\alpha$ and $\varphi_\alpha$ induce linear operators $e,f\colon M\to M$ that are homogeneous of degree $+2$ and $-2$, resp. By what we have shown above, we have $(ef-fe)|_{M_l}=[l]_\alpha\id_{M_l}$. We are hence in the situation of Lemma \ref{lemma-comrelgen} in the appendix (replace the variable $v$ in Lemma \ref{lemma-comrelgen} by the variable $v_\alpha$). From this lemma we deduce
$$
\epsilon_{\alpha}^m\varphi_\alpha^n(v)=\sum_{r=0}^{\min(m,n)}\frac{[m]_\alpha^![n]_\alpha^!}{[m-r]_\alpha^![n-r]_\alpha^!} \qchoose{\lgl\mu,\alpha^\vee\rgl+m-n}{r}_\alpha\varphi_\alpha^{n-r}\epsilon_\alpha^{m+r}(v),
$$
 or 
$$
\epsilon_{\alpha}^{[m]}\varphi_\alpha^{[n]}(v)=\sum_{r=0}^{\min(m,n)} \qchoose{\lgl\mu,\alpha^\vee\rgl+m-n}{r}_\alpha\varphi_\alpha^{[n-r]}\epsilon_\alpha^{[m-r]}(v)
$$
for all $v\in M_\mu$ and $m,n>0$.
\end{proof}

\subsection{A bilinear form} 
For an arbitrary path $\ul\gamma=(\gamma_1,\dots,\gamma_l)$ define 
\begin{align*}
\epsilon_{\ul\gamma}&:=\epsilon_{\gamma_1}\circ\dots\circ\epsilon_{\gamma_l},\\
\varphi_{\ul\gamma}&:=\varphi_{\gamma_1}\circ\dots\circ\varphi_{\gamma_l}.
\end{align*}
These are  $\SQ$-linear operators on $P_\SQ(\lambda)$ of degree $+\height(\ul\gamma)$ and $-\height(\ul\gamma)$, resp. 

For a path $\ul\delta=(\delta_1,\dots,\delta_l)$ let $\ul\delta^r:=(\delta_l,\delta_{l-1},\dots,\delta_1)$ be the reversed path. Then  $l(\ul\delta^r)=l(\ul\delta)$ and $\height(\ul\delta^r)=\height(\ul\delta)$.
If $\ul\delta$ and $\ul\gamma$ have the same height, then $\epsilon_{\ul\delta^r}(\ul\gamma)$ is a $\SQ$-multiple of the empty path. We identify $\SQ=\SQ\ul\emptyset$, so that the following definition makes sense.

\begin{definition} We denote by $b_\lambda\colon P_\SQ(\lambda)\times P_\SQ(\lambda)\to\SQ$ the bilinear form defined by
$$
b_\lambda(\ul\delta,\ul\gamma):=
\begin{cases}
0,&\text{ if $\height(\ul\delta)\ne\height(\ul\gamma)$},\\
\epsilon_{\ul\delta^r}(\ul\gamma),&\text{ if $\height(\ul\delta)=\height(\ul\gamma)$}
\end{cases}
$$
for paths $\ul\delta$, $\ul\gamma$. 
\end{definition}
\begin{proposition}
\begin{enumerate}
\item The bilinear form $b_\lambda$ is symmetric.
\item For $\mu\ne\mu^\prime$, the weight spaces $P_\SQ(\lambda)_\mu$ and $P_\SQ(\lambda)_{\mu^\prime}$ are $b_\lambda$-orthogonal.
\item For each path  $\ul\delta$, the $\epsilon_{\ul\delta}$ is biadjoint to the operator $\varphi_{\ul\delta^r}$.  
\end{enumerate}
\end{proposition} 
\begin{proof} Statement (2) is clear from the definition. Statement (1) is the quantum analogue of Lemma 3.9 in \cite{LefOp} and is proven in an analogous way. In order to prove (3), by induction we only need to consider the case $\ul\delta=(\alpha)$ with $\alpha\in\Pi$. Again, Lemma 3.9 in \cite{LefOp} contains the corresponding statement in the non-quantum world. The arguments in the quantum world are analogous. 
\end{proof}

\subsection{A lattice inside $P_\SQ(\lambda)$}
For any path $\ul\delta$  define a scalar  $\ul\delta^!\in\SZ\setminus\{0\}$ by induction on the length. We set $\ul\emptyset^!=1$. Suppose that $l\ge 1$ and consider $\ul\delta=(\delta_1,\dots,\delta_l)$. Set $\alpha:=\delta_1$ and $s:=\max\{i\in\{1,\dots,l\}\mid \alpha=\delta_1=\delta_2=\dots=\delta_i\}$. Now set
$$
\ul\delta^!:=[s]^!_\alpha\cdot {\ul\delta^{\prime}}^!,
$$
where $\ul\delta^\prime=(\delta_{s+1},\dots,\delta_l)$. 
For example, for $\ul\delta=(\alpha,\alpha,\beta,\gamma,\beta,\beta,\beta)$ with $\alpha\ne\beta$ and $\beta\ne\gamma$ we have $\ul\delta^!=[2]^!_\alpha[1]^!_\beta[1]^!_\gamma[3]^!_\beta$. Note that  $\ul\delta^!=\ul\delta^{r!}$. 
Now set
$$
\lgl\ul\delta\rgl:=\frac{1}{\ul\delta^!}\ul\delta\in P_\SQ(\lambda).
$$
Let  $P_\SZ(\lambda)$ be the $\SZ$-module  generated inside $P_\SQ(\lambda)$ by the elements $\lgl \ul\delta\rgl$ for all paths $\ul\delta$. Then $P_\SZ(\lambda)$ inherits an $X$-grading $P_\SZ(\lambda)=\bigoplus_{\mu\in X}P_{\SZ}(\lambda)_\mu$, and the paths $\lgl\ul\delta\rgl$ with $\height(\ul\delta)=\lambda-\mu$ form a basis of $P_\SZ(\lambda)_\mu$ as a $\SZ$-module.


\begin{lemma} \label{lemma-factorials} Let $\alpha\in\Pi$ and $n\ge 0$. Then the operators  $\epsilon_{\alpha}^{[n]}$ and $\varphi_{\alpha}^{[n]}$ stabilize $P_\SZ(\lambda)$.
\end{lemma}

\begin{proof} Let  $\ul\delta=(\delta_1,\dots,\delta_l)$ be a path. We want to show that $\epsilon_{\alpha}^{[n]}(\lgl\ul\delta\rgl)$ and $\varphi_{\alpha}^{[n]}(\lgl\ul\delta\rgl)$ are contained in $P_\SZ(\lambda)$. This is clear in the case $\ul\delta=\ul\emptyset$, so we can assume $l\ge 1$. Set $\beta:=\delta_1$ and let $s\ge 1$ be such that $\beta=\delta_1=\dots=\delta_s\ne\delta_{s+1}$. Hence $\lgl\ul\delta\rgl=\varphi_{\beta}^{[s]}(\lgl\ul\delta^\prime\rgl)$ with $\ul\delta^\prime=(\delta_{s+1},\dots,\delta_l)$.

We begin with showing the $\varphi_{\alpha}^{[n]}(\lgl\ul\delta\rgl)\in P_\SZ(\lambda)$. First, suppose that $\alpha\ne\beta$. Then $\varphi_{\alpha}^{[n]}(\lgl\ul\delta\rgl)=\lgl\varphi_\alpha^n(\lgl\ul\delta\rgl)\rgl$ is contained in $P_\SZ(\lambda)$. If $\alpha=\beta$, then 
\begin{align*}
\varphi_{\alpha}^{[n]}(\lgl\ul\delta\rgl)&=\varphi_{\alpha}^{[n]}\varphi_{\alpha}^{[s]}(\lgl\ul\delta^\prime\rgl)\\
&=\frac{[n+s]_\alpha^!}{[n]_\alpha^![s]_\alpha^!}\varphi_{\alpha}^{[n+s]}(\lgl\ul\delta^\prime\rgl).
\end{align*}
As $\frac{[n+s]_\alpha^!}{[n]_\alpha^![s]_\alpha^!}\in\SZ$ and $l(\ul\delta^\prime)<l(\delta)$,  we can use induction over the length and deduce that $\varphi_{\alpha}^{[n]}(\lgl\ul\delta\rgl)\in P_\SZ(\lambda)$.

Now we want to show that  $\epsilon_{\alpha}^{[n]}(\lgl\ul\delta\rgl)\in P_\SZ(\lambda)$. We have
\begin{align*}
\epsilon_{\alpha}^{[n]}(\lgl\ul\delta\rgl)&=\epsilon_{\alpha}^{[n]}\varphi_\beta^{[s]}(\lgl\ul\delta^\prime\rgl)\\
&=\sum_{r\ge 0} c_r\varphi_{\beta}^{[s-r]}\epsilon_{\alpha}^{[n-r]}(\lgl\ul\delta^\prime\rgl)
\end{align*}
for some $c_r\in\SZ$ by Lemma \ref{lemma-comrels}. By induction we can assume $\epsilon_{\alpha}^{[n-r]}(\lgl\ul\delta^\prime\rgl)\in P_{\SZ(\lambda)}$. So $\varphi_{\beta}^{[s-r]}\epsilon_{\alpha}^{[n-r]}(\lgl\ul\delta^\prime\rgl)\in P_\SZ(\lambda)$ by what we have already shown above. It follows that $\epsilon_{\alpha}^{[n]}(\lgl\ul\delta\rgl)\in P_\SZ(\lambda)$. 
\end{proof}

For a path $\ul\gamma$ set  $\varphi_{\lgl \ul\gamma\rgl}=\frac{1}{\ul\gamma^!}\varphi_{\ul\gamma}$ and $\epsilon_{\lgl \ul\gamma\rgl}=\frac{1}{\ul\gamma^!}\epsilon_{\ul\gamma}$. Note that these operators are compositions of various $\varphi_{\alpha}^{[n]}$'s and $\epsilon_{\alpha}^{[n]}$'s, hence they stabilize $P_\SZ(\lambda)\subset P_\SQ(\lambda)$.

\begin{lemma} \label{lemma-bstab}The restriction of $b_\lambda$ to $P_\SZ(\lambda)\times P_\SZ(\lambda)$ takes values in $\SZ$.
\end{lemma}

\begin{proof} For paths $\ul\delta$ and $\ul\gamma$ of the same height we have 
\begin{align*}
(\lgl\ul\delta\rgl,\lgl\ul\gamma\rgl)&=\ul\delta^{!-1}\ul\gamma^{!-1}(\ul\delta,\ul\gamma)\\
&=\ul\delta^{!-1}\ul\gamma^{!-1}\epsilon_{\ul\delta^r}(\ul\gamma)\\
&=\epsilon_{\lgl\ul\delta^r\rgl}(\lgl\ul\gamma\rgl)\in\SZ\ul\emptyset,
\end{align*}
as $P_\SZ(\lambda)$ is stable under the operator  $\epsilon_{\lgl\ul\delta^r\rgl}$.
\end{proof}
\subsection{Base change and the radical of $b_\lambda$} Now let $\SK$  again be a field, and $q\in\SK^\times$ an invertible element. As before we view $\SK$ as a $\SZ$-algebra. 
Set 
$$
P_{\SK}(\lambda):=P_\SZ(\lambda)\otimes_\SZ \SK.
$$
For any path $\ul\gamma$ we denote by the same symbols $\epsilon_{\lgl\ul\gamma\rgl}$ and $\varphi_{\lgl\ul\gamma\rgl}$  the induced $\SK$-linear operators on $P_\SK(\lambda)$, and we denote by $b_\lambda$  the bilinear form on $P_{\SK}(\lambda)$  induced by $(\cdot,\cdot)_\SZ$. 
Let
$
\rad_\SK(\lambda)\subset P_\SK(\lambda)$ be the radical of the bilinear form $b_\lambda$. Then $\rad_\SK(\lambda)$ is a graded subspace in $P_\SK(\lambda)$, i.e. $\rad_\SK(\lambda)=\bigoplus_{\mu\in X}\rad_\SK(\lambda)_\mu$.

\begin{lemma}  The radical $\rad_\SK(\lambda)$  is stable under the operators $\epsilon_{\lgl\ul\gamma\rgl}$ and $\varphi_{\lgl\ul\gamma\rgl}$ for all paths $\ul\gamma$. 
\end{lemma}
\begin{proof}
This follows immediately from the fact that $\epsilon_{\lgl\ul\gamma\rgl}$ and  $\varphi_{\lgl\ul\gamma\rgl}$ are biadjoint with respect to $b_\lambda$. 
\end{proof}

We set
$$
S_\CK(\lambda):=P_\SK(\lambda)/\rad_\SK(\lambda).
$$
This is an $X$-graded space with induced operators $\epsilon_{\lgl\ul\gamma\rgl}$ and $\varphi_{\lgl\ul\gamma\rgl}$ for all paths $\ul\gamma$. Moreover, it carries a symmetric bilinear form $b_\lambda$ that is non-degenerate. We denote the linear operators induced by $\epsilon_{\alpha}^{[m]}$ and $\varphi_{\beta}^{[n]}$ by $E_{\alpha,m}$ and $F_{\beta,n}$, resp. 

\begin{proposition} \label{prop-propL} The $\SK$-vector space $S_{\SK}(\lambda)$ together with the induced $X$-grading and the set of operators $E_{\alpha,n}$,  $F_{\alpha,n}$ with $\alpha\in\Pi$,  $n>0$ is an object in $\CC_{\SK}$. It has the following properties.
\begin{enumerate}
\item $S_{\SK}(\lambda)_\lambda$ is one dimensional, and $S_{\SK}(\lambda)_\mu\ne 0$ implies $\mu\le \lambda$.
\item $\End_{\CC_{\SK}}(S_{\SK}(\lambda))=\SK\cdot \id_{S_{\SK}(\lambda)}$. In particular, an endomorphism on $S_{\SK}(\lambda)$ is an automorphism if and only if its $\lambda$-component is an automorphism.
\end{enumerate}
\end{proposition}
\begin{proof} In order to show that $S_{\SK}(\lambda)$ is an object in $\CC_\SK$, we need to show that it  satisfies properties (C1), (C2) and (C3). By construction, $S_\SK(\lambda)_\mu\ne 0$ implies $\mu\le\lambda$, so property (C1) holds. Property (C2) follows from the commutation relations  that we have proven in Lemma \ref{lemma-comrels}.  So we need to verify (C3). As $\lambda$ is the highest weight of $S_{\SK}(\lambda)$ we deduce $\im F_\lambda=0$ and $\ker E_\lambda=S_{\SK}(\lambda)_\lambda$, so $S_\SK(\lambda)_\lambda=\im F_\lambda\oplus\ker E_\lambda$. Now suppose $\mu<\lambda$. As the elements $\varphi_{\lgl\ul\gamma\rgl}(\ul\emptyset)$, where $\ul\gamma$ runs over all paths, form a basis of $P_{\SK}(\lambda)$, and hence a generating set for $S_{\SK}(\lambda)$, we have $\im F_\mu=S_{\SK}(\lambda)_\mu$. So it remains to show that $\ker E_\mu=0$. So let $v\in\ker E_\mu$. Then $b_\lambda(v,F_{\alpha,m}(w))=b_\lambda(E_{\alpha,m}(v),w)=0$ for all $\alpha\in\Pi$, $m>0$ and $w\in S_\SK(\lambda)_{\mu+m\alpha}$. Hence $\ker E_\mu$ is $b_\lambda$-orthogonal to $\im F_\mu$. As $\im F_\mu=S_\SK(\lambda)_\mu$ and as $b_\lambda$ is non-degenerate and the weight space decomposition is orthogonal, we deduce $\ker E_\mu=0$.  So also property (C3) is satisfied.  Hence $S_{\SK}(\lambda)$  is indeed an object in $\CC_\SK$. The properties stated in (1) follow immediately from the construction.

So let us show that property (2) holds. Let $f$ be an endomorphism on $S_{\SK}(\lambda)$. As $S_{\SK}(\lambda)_\lambda$ is one-dimensional, we can identify $f_\lambda$ with an element in $\SK$. We set $g:=f-f_\lambda\id_{S_{\SK}(\lambda)}$. This now is an endomorphism that vanishes on $S_{\SK}(\lambda)_\lambda$. But then $g(\varphi_{\lgl\ul\gamma\rgl}(\ul\emptyset))=\varphi_{\lgl\ul\gamma\rgl}g(\ul\emptyset)=0$. As $S_{\SK}(\lambda)$ is generated as a $\SK$-vector space by the elements $\varphi_{\lgl\ul\gamma\rgl}(\ul\emptyset)$ we deduce $g=0$, hence $f=f_\lambda\cdot \id_{S_{\SK}(\lambda)}$. 
\end{proof}

\subsection{Semisimplicity of $\CC_\SK$}

Now we show that, essentially, we have already constructed all objects in $\CC_\SK$.

\begin{proposition}\label{prop-semisimple} Let $M$ be an object in $\CC_{\SK}$ and suppose that $M_\mu$ is finite dimensional for all $\mu\in X$. Then there exists an index set $I$ and weights $\lambda_i\in X$ for $i\in I$ such  that $M\cong S_\SK(\lambda_1)\oplus\dots\oplus S_\SK(\lambda_l)$. 
\end{proposition}
\begin{proof} By property (C1) there is a maximal element $\lambda\in X$ such that $M_\lambda\ne 0$. We set $\CJ:=\{\mu\in X\mid \lambda\le\mu\}$. Note that this is an open subset of $X$. Set $n=\dim M_\lambda$. Then we can choose a vector space isomorphism $f_{(\lambda)}\colon S_\SK(\lambda)_\lambda^{\oplus n}\xrightarrow{\sim} M_\lambda$. We denote by $f_\mu\colon S_\SK(\lambda)_\mu\to M_\mu$ the zero homomorphism for all $\mu\in\CJ$, $\mu\ne \lambda$ (both sides are $0$ anyways). Then we can apply Lemma \ref{lemma-extmor} and deduce that there is a morphism $f\colon S_\SK(\lambda)^{\oplus n}\to M$ in $\CC_\SK$ with $f_\lambda=f_{(\lambda)}$. In the same way we obtain an extension $\tilde f\colon M\to S_\SK(\lambda)^{\oplus n}$ of $f_{(\lambda)}^{-1}\colon M_\lambda\xrightarrow{\sim} S_\SK(\lambda)^{\oplus n}_\lambda$. The composition $\tilde f\circ f$ is an endomorphism of $S_\SK(\lambda)^{\oplus n}$ that restricts to the identity on the highest weight space. Proposition \ref{prop-propL} implies that $\tilde f\circ f$ is the identity. Hence $S_\SK(\lambda)^{\oplus n}$ is a direct summand of $M$, and a direct complement $N$ satisfies $N_\lambda=0$. From here we can proceed by induction. 
\end{proof}

\subsection{Connection to representation theory}

The following shows that we dealt with $X$-graded representations of $U_\SK$ all along. 
\begin{theorem}\label{thm-conrep} There exists a unique  $U_{\SK}$-module structure on $S_{\SK}(\lambda)$ such that $E_{\alpha,n}$ and $F_{\alpha,n}$ are the action maps of $e_\alpha^{[n]}$ and $f_\alpha^{[n]}$, resp., and $k_\alpha$, $k_{\alpha}^{-1}$ act on $S_\SK(\lambda)_\mu$ via the character $\mu$, for all $\alpha\in\Pi$, $n>0$ and $\mu\in X$. This makes $S_{\SK}(\lambda)$ into an object in $\CO_{\SK}$ that isomorphic to $L_{\SK}(\lambda)$. 
\end{theorem}
\begin{proof} We actually prove a reverse statement. Consider $L_{\SK}(\lambda)$ as an $X$-graded space with operators by letting $E_{\alpha,n}$ and $F_{\alpha,n}$ be the action maps of $e_{\alpha}^{[n]}$ and $f_{\alpha}^{[n]}$, resp. Then property (C1) is obviously satisfied by this data, and Lemma \ref{lemma-comrelsU} implies (C2).  As $\lambda$ is the maximal weight, we have $L_\SK(\lambda)=\ker E_\lambda$ and $\im F_\lambda=0$. As $L_{\SK}(\lambda)_\mu$ has no non-trivial primitive vectors for $\mu\ne \lambda$, we deduce (C3). So $L_\SK(\lambda)$ can be viewed as an object in $\CC_\SK$.  Proposition \ref{prop-semisimple} implies now that $L_{\SK}(\lambda)$ is isomorphic to a direct sum of copies of various $S_{\SK}(\mu)$'s. As $L_{\SK}(\lambda)$ is generated (over the $F_{\alpha,n}$-maps) by a non-zero vector of weight $\lambda$, we deduce $L_{\SK}(\lambda)\cong S_{\SK}(\lambda)$ in $\CC_{\SK}$.  This module structure is unique, as  $U_\SK$ is generated by the elements $e_\alpha^{[n]}$, $f_\alpha^{[n]}$ with $\alpha\in\Pi$, $k_\alpha$, $k_\alpha^{-1}$ for $\alpha\in\Pi$ and $n\ge 0$.
\end{proof}

\section{Cyclotomic polynomials and periodicity}

The aim of this section is to study the matrix of the bilinear forms $b_\lambda$ in terms of the basis of simple root paths. It turns out that we can view the entries of this matrix as {\em quantum polynomial functions}. 

\subsection{Quantum polynomial functions associated to paths}
Let $f\colon X\to\DZ$ be a group homomorphism, and $\alpha\in\Pi$. We denote by $[f(\cdot)]_{\alpha}\colon X\to\SZ$ the map $\lambda\mapsto [f(\lambda)]_{\alpha}= \frac{v_\alpha^{ f(\lambda)}-v_\alpha^{-f(\lambda)}}{v_\alpha-v^{-1}}$. 

\begin{definition}  We denote by $\SZ[X]$ the $\SZ$-subalgebra inside the algebra of all maps from $X$ to $\SZ$ that is generated by the maps of the form $[f(\cdot)]_{\alpha}$ for all group homomorphisms $f\colon X\to \DZ$ and $\alpha\in\Pi$. We  set $\SQ[X]:=\SZ[X]\otimes_\SZ\SQ$.
\end{definition}

For paths $\ul\delta,\ul\gamma$ of the same height define a function $\tilde a_{\ul\delta,\ul\gamma}\colon X\to \SZ$ by 
$$
\tilde a_{\ul\delta,\ul\gamma}(\lambda)=b_\lambda(\ul\delta,\ul\gamma).
$$
Then the following is immediate.

\begin{lemma} \begin{enumerate}
\item We have $\tilde a_{\ul\emptyset,\ul\emptyset}(\lambda)=1$ for all $\lambda\in X$.
\item We have $\tilde a_{\ul\delta,\ul\gamma}=\tilde a_{\ul\gamma,\ul\delta}$ for all paths $\ul\delta$, $\ul\gamma$ of the same height.
\item We have   
$$
\tilde a_{\ul\delta,\ul\gamma}(\lambda)=\sum_{i,\delta_l=\gamma_i}[\lambda-\gamma_{i+1}-\dots-\gamma_{l},\delta_1^\vee]_{\delta_1} \tilde a_{\ul\delta^{(l)}\ul\gamma^{(i)}}(\lambda)
$$
 for all paths $\ul\delta$, $\ul\gamma$ of the same height and all $\lambda\in X$.
\end{enumerate}
\end{lemma}
It follows from (1) and (3) by induction that  $\tilde a_{\ul\delta,\ul\gamma}\in\SZ[X]$. Now set
$$
a_{\ul\delta,\ul\gamma}:=\frac{1}{\ul\delta^{!}\ul\gamma^{!}}\tilde a_{\ul\delta,\ul\gamma}.
$$
Then $a_{\ul\delta,\ul\gamma}\in\SQ[X]$. For $\lambda\in X$ we have, however,
$a_{\ul\delta,\ul\gamma}(\lambda)=\frac{1}{\ul\delta^{!}\ul\gamma^{!}}b_\lambda(\ul\delta,\ul\gamma)=
b_\lambda(\lgl\ul\delta\rgl,\lgl\ul\gamma\rgl)\in\SZ$,
by Lemma \ref{lemma-bstab}, i.e. $a_{\ul\delta,\ul\gamma}$ is a function in $\SQ[X]$ with $a_{\ul\delta,\ul\gamma}(X)\subset\SZ$.

For $\nu\ge 0$ define the square matrix
$$
A_{\nu}:=\left(a_{\ul\delta\ul\gamma}\right)_{\height(\ul\delta)=\height(\ul\gamma)=\nu}
$$
with entries in $\SQ[X]$. For $\lambda\in X$ we denote by $A_\nu(\lambda)$ the matrix obtained by evaluating each entry at $\lambda$. So $A_\nu(\lambda)$ is a symmetric matrix with entries in $\SZ$. For $(\SK,q)$ as above we denote by $A_\mu(\lambda)_\SK$ the matrix with entries in $\SK$ obtained by base change $\SZ\to\SK$. It represents the restriction of the bilinear form $b_{\lambda}$ to the $\lambda-\nu$ weight space  of $P_\SK(\lambda)$ with respect to the basis $\{\lgl\ul\delta\rgl\}_{\height(\ul\delta)=\nu}$.

\begin{proposition} \label{prop-chars} For all $\lambda,\mu\in X$ we have
$$
\dim_{\SK} L_{\SK}(\lambda)_{\mu}=\rk _{\SK}\, A_{\lambda-\mu}(\lambda)_{\SK}.
$$
\end{proposition}
\begin{proof} We have $\dim_\SK L_{\SK}(\lambda)_{\mu}=\dim_\SK  S_{\SK}(\lambda)_{\mu}$ by Theorem \ref{thm-conrep}. By definition,  $S_\SK(\lambda)_\mu=P_\SK(\lambda)_\mu/\rad_\SK(\lambda)_\mu$ and hence  $\dim_\SK  S_{\SK}(\lambda)_{\mu}=\rk _{\SK}\, A_{\lambda-\mu}(\lambda)_{\SK}$. 
\end{proof}

\subsection{Cyclotomic polynomials}
For $l\ge 1$ denote by $\sigma_l\in\DZ[v]$ the $l$-th cyclotomic polynomial (so $\sigma_1=v-1$, $\sigma_2=v+1$, $\sigma_3=v^2+v+1$,...). Then $v^n-1=\prod_{l\ge1, l|n}\sigma_l$.
So
\begin{align*}
[n]&=\frac{v^{n}-v^{-n}}{v-v^{-1}}=v^{-n+1}\frac{v^{2n}-1}{v^{2}-1}\\
&=v^{-n+1}\prod_{l\ge 3,l|2n}\sigma_l.
\end{align*}

\begin{lemma} \label{lemma-per1} Let $F$ be an element in $\SZ[X]$. Then  $F(\lambda+l\gamma)\equiv F(\lambda)\mod  \sigma_l$   for all $\lambda,\gamma\in X$.
\end{lemma}
\begin{proof} 
It is sufficient to show the statement for  all $F$ of the form $[f(\cdot)]_\alpha$ for some group homomorphism $f\colon X\to\DZ$ and some $\alpha\in\Pi$. For these, we calculate as follows.  
\begin{align*}
[f(\lambda+l\gamma)]_\alpha&=[f(\lambda)+lf(\gamma)]_\alpha=\frac{v_\alpha^{f(\lambda)+lf(\gamma)}-v_\alpha^{-f(\lambda)-lf(\gamma)}}{v_\alpha-v_\alpha^{-1}}\\
&=\frac{v_\alpha^{f(\lambda)}(v_\alpha^l)^{f(\gamma)}-v_\alpha^{-f(\lambda)}(v_\alpha^{-l})^{f(\gamma)}}{v_\alpha-v_\alpha^{-1}}\\
&\equiv \frac{v_\alpha^{f(\lambda)}-v_\alpha^{-f(\lambda)}}{v_\alpha-v_\alpha^{-1}}=[f(\lambda)]_\alpha \mod\sigma_l
\end{align*}
as $v_\alpha^l=v^{d_\alpha l}=1\mod\sigma_l$ for all $\alpha\in\Pi$.
 \end{proof}

For $\nu\ge 0$ define $c_\alpha(\nu)\in \DZ_{\ge 0}$ by $\nu=\sum_{\alpha\in\Pi}c_\alpha(\nu)\alpha$.

\begin{lemma} \label{lemma-permat} Let $\nu\ge 0$ and $l\in\DN$ be such that $l>c_{\alpha}(\nu)$ for all $\alpha\in\Pi$. Then 
$$
A_\nu(\lambda+l\gamma)\equiv A_\nu(\lambda)\mod \sigma_l
$$
for all $\lambda,\gamma\in X$. 
\end{lemma}
Let $(\SK,q)$ be as before. If $\sigma_l(q)=0$, then the above Lemma means $A_\nu(\lambda+l\gamma)_\SK= A_\nu(\lambda)_\SK$.
\begin{proof} Let $s\in\SQ[X]$ be the entry in $A_\nu$ at row $\ul\delta$ and column $\ul\gamma$. Then $\tilde s:=\ul\delta^!\ul\gamma^! s\in\SZ[X]$. Lemma \ref{lemma-per1} implies that $\tilde s(\lambda+l\gamma)- \tilde s(\lambda)=\ul\delta^!\ul\gamma^!(s(\lambda+l\gamma)-s(\lambda))$ is divisible by $\sigma_l$ for all $\lambda,\gamma\in X$.  Note that the height of $\ul\delta$ and $\ul\gamma$ is $\nu$, so  our assumption on $l$ implies that each simple root $\alpha$ occurs in either path at most $l-1$ times. That implies that $\sigma_l$ divides neither $\ul\delta^!$ nor $\ul\gamma^!$.  So we deduce that $s(\lambda+l\gamma)-s(\lambda)$ is divisible by $\sigma_l$. 
\end{proof}

Let $(\SK,q)$ be as before. 
Proposition \ref{prop-chars} and  Lemma \ref{lemma-permat} immediately imply the following, which is the main result of this article.
\begin{theorem}\label{thm-periodicity} Suppose that  $\lambda,\mu\in X$ and $l\in\DN$ are such that $\mu\le \lambda$ and  $l>c_{\alpha}(\lambda-\mu)$ for all $\alpha\in\Pi$. Suppose furthermore that $\sigma_l(q)=0$. Then
$$
\dim_{\SK} L_{\SK}(\lambda)_{\mu}=\dim_{\SK} L_{\SK}(\lambda+l\nu)_{\mu+l\nu}
$$
for all $\nu\in X$. 
\end{theorem}

\section{Appendix}
This section contains a lengthy calculation that is used in the main body of the article.
Let $M=\bigoplus_{l\in\DZ}M_l$ be a $\DZ$-graded $\SQ=\DQ(v)$-module. Let $e,f\colon M\to M$ be $\SQ$-linear endomorphisms of degree $+2$ and $-2$, resp. Suppose that  for all  $l\in \DZ$ we have $(ef-fe)|_{M_{l}}=[l]\cdot\id_{M_{l}}=\frac{v^{l}-v^{-l}}{v-v^{-1}}\id_{M_l}$.

\begin{lemma}\label{lemma-comrelgen}  For all $m,n>0$ and $l\in \DZ$ we have
$$
e^{m}f^{n}|_{M_{l}}=\sum_{r=0}^{\min(m,n)}\frac{[m]^![n]^!}{[m-r]^![n-r]^!} \qchoose{l+m-n}{r}f^{n-r}e^{m-r}|_{M_{l}}.
$$
\end{lemma}

\begin{proof} We first prove this in the case $m=1$. Then the claim reads
$$
ef^n(v)=f^ne(v)+[n] [ l+1-n]f^{n-1}(v)
$$
for all $v\in M_{ l}$. For $n=1$ this is the assumed commutation relation between the operators $e$ and $f$. So suppose the claim holds for some $n$.  Let $v\in M_{ l}$. Then
\begin{align*}
 ef^{n+1}(v)&=(ef^n)f(v)\\
 &=f^nef(v)+[n][ l-2+1-n]f^{n}(v)\quad\text{ (as $f(v)\in M_{ l-2}$)}\\
 &=f^{n+1}e(v)+[l]f^{n}(v) +[n][ l-1-n]f^{n}(v)\quad\text{ (as $[e,f](v)=[l]v$)},\\
\end{align*}
and the claim boils down to showing that
$$
[ l]+[n][ l-1-n]= [n+1][ l-n],
$$
which follows from Lemma \ref{lemma-easypeasy} (set $a=n+1$,  $b= l-n$, $c=1$). 
 So we have proven the claim in the case $m=1$ and arbitrary $n>0$. 

Now we fix $n$ and prove the formula  by induction on $m$. For $m=1$ this is settled already.  So suppose the above equation holds for some $m$. Let $v\in M_{ l}$. 
We calculate, using the induction hypothesis and setting $\zeta:= l+m-n$,
\begin{align*}
e^{m+1}f^n(v)&=e(e^{m}f^n)(v)\\
&=e(\sum_{r=0}^{\min(m,n)}\frac{[m]^![n]^!}{[m-r]^![n-r]^!}\qchoose{\zeta}{r}f^{n-r}e^{m-r}(v))\\
&=\sum_{r=0}^{\min(m,n)}\frac{[m]^![n]^!}{[m-r]^![n-r]^!}\qchoose{\zeta}{r}ef^{n-r}e^{m-r}(v).
\end{align*}
Now fix, for a moment, a number $r$ with $0\le r\le \min(m,n)$. Set $w_r:=e^{m-r}(v)$. This is an element in $M_{ l+2(m-r)}$. Then the already proven case $m=1$ yields
\begin{align*}
ef^{n-r}(w_r)&=f^{n-r}e(w_r)+[n-r][ l+2(m-r)+1-(n-r)]f^{n-r-1}(w_r)\\
&= f^{n-r}e(w_r)+[n-r][\zeta+m-r+1]f^{n-r-1}(w_r).
\end{align*}
We obtain $e^{m+1}f^n(v)=\sum_{0\le s\le\min(m+1,n)} c_s f^{n-s}e^{m+1-s}(v)$ with
\begin{align*}
c_s&=\frac{[m]^![n]^!}{[m-s]^![n-s]^!}\qchoose{\zeta}{s}+\\&\quad+\frac{[m]^![n]^![n-(s-1)][\zeta+m-(s-1)+1]}{[m-(s-1)]^![n-(s-1)]^!}\qchoose{\zeta}{s-1}
\end{align*}
for $0\le s\le \min(m,n)$ and, in the case  $n\ge m+1$, 
\begin{align*}
c_{m+1}&=\frac{[m]^![n]^![n-m][\zeta+1]}{[n-m]^!}\qchoose{\zeta}{m} \\
&= \frac{[m+1]^![n]^!}{[n-(m+1)]^!}\qchoose{\zeta+1}{m+1}.
\end{align*}
For $c_{m+1}$ we immediately see that this equals the coefficient of $f^{n-(m+1)}$ on the left hand side of the equation that we want to prove.

 We now fix $s$ with $0\le s\le\min(m,n)$ and write $c_s=\frac{[m]^![n]^!}{[m-s]^![n-s]^!} d_s$ with
\begin{align*}
d_s&=\qchoose{\zeta}{ s}+\frac{[n-(s-1)][\zeta+m-s+2]}{[m-s+1][n-s+1]}\qchoose{\zeta}{ s-1}\\
&=\qchoose{\zeta}{ s}+\frac{[\zeta+m-s+2]}{[m-s+1]}\qchoose{\zeta}{ s-1}.
\end{align*}

Now the claim is equivalent to showing that  $c_s=\frac{[m+1]^![n]^!}{[m+1-s]^![n-s]^!}\	\qchoose{\zeta+1}{ s}$ or $d_s=\frac{[m+1]}{[m-s+1]}\qchoose{\zeta+1}{ s}$, i.e.
$$
\qchoose
{\zeta}{ s}+\frac{[\zeta+m-s+2]}{[m-s+1]}\qchoose{\zeta}{ s-1}=\frac{[m+1]}{[m-s+1]}\qchoose{\zeta+1}{  s}
$$
or 
$$
[m-s+1]\qchoose
{\zeta}{ s}+[\zeta+m-s+2] \qchoose{\zeta}{ s-1}=[m+1]\qchoose{\zeta+1}{  s}
$$
which follows from Lemma \ref{lemma-easypeasy} ($a=m+1$, $b=\zeta+1$, $c=s$). 
\end{proof}

\begin{lemma} \label{lemma-easypeasy} Let $a,b\in\DZ$ and $c\ge0$. Then we have
$$
[a]\qchoose{b}{c}=[a-c]\qchoose{b-1}{c}+[a+b-c]\qchoose{b-1}{c-1}.
$$
\end{lemma}
\begin{proof}
We multiply both sides with $[c]^!$ and arrive at the equivalent equation
$$
[a][b]\cdots[b-c+1]=[a-c][b-1]\cdots[b-c]+[a+b-c][c][b-1]\cdots[b-c+1].
$$
Both sides are divisible by $[b-1]\cdots[b-c+1]$, hence the above follows from the equation
$$
[a][b]=[a-c][b-c]+[a+b-c][c].
$$
Multiplying this equation by $(v-v^{-1})^2$ we obtain the equivalent equation 
$$
(v^{a}-v^{-a})(v^{b}-v^{-b})=(v^{a-c}-v^{-(a-c)})(v^{b-c}-v^{-(b-c)})+(v^{a+b-c}-v^{-(a+b-c)})(v^c-v^{-c})
$$
The right hand side is
$$
v^{a+b-2c}-v^{a-b}-v^{-a+b}+v^{-a-b+2c}+v^{a+b}-v^{a+b-2c}-v^{-a-b+2c}+v^{-a-b},
$$
which simplifies to  $-v^{a-b}-v^{-a+b}+v^{a+b}+v^{-a-b}$. This equals the left hand side. 
\end{proof}

\end{document}